\documentclass{tac}  
\usepackage[all,2cell]{xy}
\usepackage{mathrsfs,amssymb,stmaryrd,bbm,mathtools}

\newtheorem{theo}{Theorem}[section]

\newtheorem{coro}[theo]{Corollary}
\newtheoremrm{defi}[theo]{Definition}
\newtheoremrm{rem}[theo]{Remark}

\def\mathrmdef#1{\expandafter\def\csname#1\endcsname{{\rm#1}}}
\def\mathsfdef#1{\expandafter\def\csname#1\endcsname{{\rm\sf#1}}}
\def\mathcaldef#1{\expandafter\def\csname#1\endcsname{{\mathcal#1}}}
\mathrmdef{id}\mathrmdef{op}
\mathsfdef{Alg} \mathrmdef{Ps} \mathsfdef{Set} \mathsfdef{CAT}
\mathrmdef{Id}\mathrmdef{colim}\mathsfdef{Cat}

\def\AAA{\mathbb{A}}
\def\BBB{\mathbb{B}}
\def\DDD{\mathbb{D}}
\def\SSS{\mathbb{S}}
\def\MMM{\mathbb{M}}
\def\NNN{\mathbb{N}}

\def\alg{\mathfrak{alg}}

\def\tttt{\mathfrak{t}}

\def\bbb{\mathfrak{B}}

\def\F{\mathcal{F}}

\def\TTTTT{\mathcal{T}}

\begin{document}  
\title{Pseudoalgebras and non-canonical isomorphisms} 
\author{Fernando Lucatelli Nunes}
\address{CMUC, Department of Mathematics, University of Coimbra, 3001-501 Coimbra, Portugal}
\eaddress{fernandolucatelli@mat.uc.pt}


\thanks{Research partially supported by the Centre for Mathematics of the
University of Coimbra -- UID/MAT/00324/2013, funded by the Portuguese
Government through FCT/MCTES and co-funded by the 
European Regional Development Fund through the Partnership Agreement PT2020.}

\maketitle 
\begin{abstract}
Given a pseudomonad $\TTTTT $, we prove that a lax $\TTTTT$-morphism between pseudoalgebras
is a $\TTTTT$-pseudomorphism
if and only if there is a suitable (possibly non-canonical) invertible $\TTTTT $-transformation. This result 
encompasses several results on \textit{non-canonical isomorphisms}, including 
Lack's result on normal monoidal functors between braided monoidal categories, 
since it is applicable in any $2$-category of pseudoalgebras, such as the $2$-categories of 
 monoidal categories, cocomplete categories, bicategories, pseudofunctors and so on.
\end{abstract}

\section*{Introduction}
The problem of \textit{non-canonical isomorphisms}  consists of investigating whether, in a given situation, 
the existence of an invertible non-canonical transformation implies that a previously given canonical one is invertible
as well. In order to give a glimpse of our scope, we give some examples.

The first precise example is related to the study of preservation of colimits. Given any functor $F : \AAA\to \BBB $, assuming the existence of $\SSS$-colimits, there is an induced canonical natural transformation, induced by the image of the universal cocone of $D$ by $F$,  
$\colim _ \BBB   FD\longrightarrow F (\colim _ \AAA D) $ in $\BBB $, in which
$$\colim _ \BBB : \Cat \left[ \SSS , \BBB\right]\to \BBB , \qquad \colim _ \AAA : \Cat \left[ \SSS , \AAA\right]\to \AAA $$
are the functors that give the $\SSS $-colimits. We say that \textit{$F $ preserves $\SSS $-colimits} whenever this
canonical transformation is invertible. In this context, the problem of non-canonical isomorphisms, studied by Caccamo and Winskel~\cite{MR2207108}, is
to investigate under which conditions the existence of a natural isomorphism $\colim _ \BBB   FD\cong F (\colim _ \AAA D) $ implies that $F$
preserves $\SSS $-colimits. For instance, in the case of finite coproducts, \cite{MR2207108} proves that a functor $F$ preserves them if and only if there is a (possibly non-canonical) natural 
isomorphism as above and $F$ preserves initial objects.

Three other examples of non-canonical isomorphisms are given in \cite{MR2864762}. Namely:
\begin{enumerate}
\item Characterization of distributive categories: a category $\DDD $ with finite coproducts and products is distributive if, 
given any object $x $, the functor $x\times - : \DDD\to \DDD $ preserves finite coproducts. In this case, \cite{MR2864762} proves that the existence of an invertible 
natural transformation 
$\delta _ {{}_{(x, y, z)}}: (x\times y)\sqcup (x\times z)\longrightarrow x\times (y\sqcup z)  $   
implies that $\DDD $ is distributive.
\item Characterization of semi-additive categories: a category $\BBB $ with finite products and coproducts is semi-additive if (1) it is pointed and (2) the canonical natural transformation $\psi: -\sqcup -\longrightarrow -\times - $ induced by the identities and zero morphisms is invertible. In \cite{MR2864762},  it is shown that the existence of any natural isomorphism   $-\sqcup -\longrightarrow -\times - $ implies that $\AAA $ is semi-additive.
\item Braided monoidal categories: Lack proved that, in the presence of a suitably defined invertible non-canonical isomorphism, 
a normal monoidal functor is actually a strong monoidal functor. This result encompasses the common part of 
both situations above. 
\end{enumerate}

Two-dimensional monad theory~\cite{MR1673316, MR1764104} gives a unifying approach to study several aspects of
two-dimensional universal algebra~\cite{MR0360749, MR1007911,  MR3491845}. 
This fact is illustrated by the various examples of $2$-categories of (lax-/pseudo)algebras 
in the literature.
For this reason, 
results in two-dimensional monad theory usually gives light to a wide range of situations, 
having many applications, \textit{e.g.} \cite{MR1776428, MR3318158, 2016arXiv160604999L}. 

The aim of this note is to frame the problem of non-canonical isomorphisms in the context of $2$-dimensional monad theory: we show
that, given a pseudomonad $\TTTTT $, a lax $\TTTTT $-morphism $\mathtt{f} $ is a $\TTTTT $-pseudomorphism if and only if there is a suitable 
(possibly ``non-canonical'') invertible $\TTTTT $-transformation as defined in \ref{noncanonicalisomorphism}.
This result
encompasses the four situations above, generalizing Lack's result on strong monoidal functors and being applicable to study  analogues in several other instances, including results on the
$2$-categories of monoidal categories, pseudofunctors, cocomplete categories, categories with certain types of
colimits and any
other example of $2$-category of pseudoalgebras and lax morphisms.

In Section \ref{Basic} we fix terminology, giving basic definitions and known results on $2$-dimensional monad theory. Section \ref{Prenormal} gives the main result of this note. Finally, Section \ref{Applications} gives brief comments on particular cases, showing how all the 
situations above are encompassed by our main theorem, how we can study further examples (such as the case of monoidal functors between monoidal categories) and how the situation is simplified for Kock-Z\"{o}berlein pseudomonads (which includes, in particular, the cases of preservation of weighted colimits).

\section{Basics}\label{Basic}

In order to fix notation, we give basic definitions in this section. We also give some results assumed in Section \ref{Prenormal}.
Our setting is the tricategory $2\textrm{-}\CAT$ of $2$-categories, pseudofunctors, pseudonatural transformations
and modifications. We follow the notation and definitions of Section 2 of \cite{MR3491845}. 
We give the definitions of pseudomonads~\cite{MR1673316, MR1764104} and the associated $2$-category of lax algebras~\cite{2016arXiv160703087L} below.

\begin{defi}[Pseudomonad]
A \textit{pseudomonad} $\TTTTT $ on a $2$-category $\bbb $ consists of a sextuple $(\TTTTT,  m  , \eta , \mu,   \iota, \tau )$, in which $\TTTTT:\bbb\to\bbb $ is a pseudofunctor, $ m  : \TTTTT^2\longrightarrow \TTTTT,  \eta : \Id _ {{}_\bbb }\longrightarrow \TTTTT$ are pseudonatural transformations and
$\tau : \Id _ {{}_{\TTTTT}}\Longrightarrow (m)(\TTTTT\eta ) $, $\iota : (m)(\eta\TTTTT )\Longrightarrow \Id _ {{}_{\TTTTT}}$, $\mu : m\left(\TTTTT m\right)\Longrightarrow m\left( m\TTTTT \right)  $
are invertible modifications satisfying coherence equations~\cite{2016arXiv160703087L, MR1673316, MR1764104}.
\end{defi}

\begin{defi}[Lax algebras]\label{laxalgebras}
Let $\TTTTT = (\TTTTT,  m  , \eta , \mu,   \iota, \tau )$ be a pseudomonad on $\bbb $. We define the $2$-category $\mathsf{Lax}\textrm{-}\TTTTT\textrm{-}\Alg _\ell $
as follows:
\begin{itemize}
\renewcommand\labelitemi{--}
\item Objects: \textit{lax $\TTTTT$-algebras} are defined by $\mathtt{z}= (Z, \alg _ {{}_{\mathtt{z}}}, \overline{{\underline{\mathtt{z}}}}, \overline{{\underline{\mathtt{z}}}}_0 )$ in which $\alg _ {{}_{\mathtt{z}}}: \TTTTT Z\to Z $
is a morphism of $\bbb $ and $\overline{{\underline{\mathtt{z}}}}:\alg _ {{}_{\mathtt{z}}}\TTTTT(\alg _ {{}_{\mathtt{z}}})\Rightarrow \alg _ {{}_{\mathtt{z}}}m_{{}_{Z}}, \overline{{\underline{\mathtt{z}}}}_0: \Id _ {{}_{Z}}\Rightarrow\alg _ {{}_{\mathtt{z}}}\eta _ {{}_{Z}}  $ are $2$-cells of $\bbb $ satisfying the coherence axioms:
$$\xymatrix@C=4em{ 
\TTTTT ^3 Z\ar[r]^{\TTTTT ^2 (\alg _ {{}_{\mathtt{z}}} ) }\ar[rd]|-{\TTTTT (m _ {{}_{Z}}) }\ar[d]_{m_{{}_{\TTTTT Z}} }
&
\TTTTT ^2 Z\ar@{}[d]|{\xLeftarrow{\widehat{\TTTTT (\overline{{\underline{\mathtt{z}}}} )} }}\ar[rd]^{\TTTTT (\alg _ {{}_{\mathtt{z}}} )}
&\ar@{}[rdd]|{=}
&
\TTTTT ^3 Z\ar[r]^{\TTTTT ^2(\alg _ {{}_{\mathtt{z}}}) }\ar@{}[rd]|{\xLeftarrow{m _ {{}_{\alg _ {{}_{\mathtt{z}}} }}^{-1}  }}\ar[d]_{m _ {{}_{\TTTTT Z }} }
&
\TTTTT^2 Z\ar[d] ^{m _ {{}_{Z}} }\ar[rd]^{\TTTTT (\alg _ {{}_{\mathtt{z}}}) }
&
\\
\TTTTT^2 Z\ar@{}[r]|-{\xLeftarrow{\mu _ {{}_{ Z }}  }}\ar[rd]_-{m _ {{}_{Z}} }
&
\TTTTT ^2 Z\ar[d]_-{m _ {{}_{Z}} }\ar[r]^{\TTTTT (\alg _ {{}_{\mathtt{z}}} ) }\ar@{}[rd]|{\xLeftarrow{\overline{{\underline{\mathtt{z}}} }  }}
&
\TTTTT Z\ar[d]^{\alg _ {{}_{\mathtt{z}}} }
&
\TTTTT ^2 Z\ar[r]_-{\TTTTT (\alg _ {{}_{\mathtt{z}}} ) }\ar[rd]_-{m _ {{}_{Z}} }
&
\TTTTT Z\ar@{}[d]|{\xLeftarrow{\overline{{\underline{\mathtt{z}}} }  } } \ar@{}[r]|{\xLeftarrow{\overline{{\underline{\mathtt{z}}} }}  }\ar[rd]|-{\alg _ {{}_{\mathtt{z}}} }
&
\TTTTT Z\ar[d]^{\alg _ {{}_{\mathtt{z}}} }
\\
&
\TTTTT Z\ar[r]_ {\alg _ {{}_{\mathtt{z}}} }
&
Z
&
&
\TTTTT Z\ar[r]_ {\alg _ {{}_{\mathtt{z}}} }
&
Z
}$$
in which 
$\widehat{\TTTTT (\overline{{\underline{\mathtt{z}}}} )}:= \left( \tttt _{{}_{( \alg _ {{}_{\mathtt{z}}})  ( m _{{}_{Z}} ) }}\right)^{-1}\left(\TTTTT (\overline{{\underline{\mathtt{z}}}})\right) \left( \tttt _{{}_{( \alg _ {{}_{\mathtt{z}}} ) (\TTTTT (\alg _ {{}_{\mathtt{z}}})  ) }}\right)$ and the $2$-cells
$$\xymatrix{ 
\TTTTT Z\ar[rrr]^{\alg _ {{}_{\mathtt{z}}}}\ar[rd]^{\eta _ {{}_{\TTTTT Z}} }\ar@{=}[dd]
&
&
&
Z\ar[ld]_-{\eta _ {{}_{Z}} }\ar@{=}[dd]
&
\TTTTT Z\ar@{=}[rr]\ar@{=}[dd]\ar[rd]^{\TTTTT (\eta _ {{}_ {Z}} )}
&
&
\TTTTT Z\ar@{}[ld]|-{\xLeftarrow{\widehat{\TTTTT (\overline{{\underline{\mathtt{z}}}}_0) }} }\ar@{=}[d]
\\
\ar@{}[r]|-{\xLeftarrow{\iota _ {{}_{Z}} } }
&
\TTTTT ^2 Z\ar@{}[ru]|-{\xLeftarrow{\eta _ {{}_{\alg _ {{}_{\mathtt{z}}}}}^{-1} } }\ar[r]_{\TTTTT (\alg _ {{}_{\mathtt{z}}} ) }\ar[ld]^{m _ {{}_{Z}} }
\ar@{}[rd]|-{\xLeftarrow {\overline{{\underline{\mathtt{z}}} } } }
&
\TTTTT Z\ar@{}[r]|-{\xLeftarrow{\overline{{\underline{\mathtt{z}}} }_0 } }\ar[rd]_-{\alg _ {{}_{\mathtt{z}}} }
&
&
\ar@{}[r]|-{\xLeftarrow{\tau _ {{}_{Z}}^{-1} } }
&
\TTTTT ^2 Z\ar[r]_-{\TTTTT (\alg _ {{}_{\mathtt{z}}}) }\ar[ld]^- {m _{{}_{Z}} }\ar@{}[rd]|-{\xLeftarrow{\overline{{\underline{\mathtt{z}}} }} }
&
\TTTTT Z\ar[d]^-{\alg _ {{}_{\mathtt{z}}}}
\\
\TTTTT Z\ar[rrr]_{\alg _ {{}_{\mathtt{z}}} }
&
&
&
Z
&
\TTTTT Z\ar[rr]_-{\alg _ {{}_{\mathtt{z}}}}
&&
Z
}$$
are identities in which $\widehat{\TTTTT (\overline{{\underline{\mathtt{z}}}}_0 )} := \left(\tttt _ {{}_{( \alg _ {{}_{\mathtt{z}}} )(\eta _ {{}_{Z}} ) }}\right)^{-1} \left(\TTTTT (\overline{{\underline{\mathtt{z}}}}_0)  \right) 
\left( \tttt _ {{}_{\TTTTT Z}} \right)$.  

Recall that a lax $\TTTTT$-algebra $(Z, \alg _ {{}_{\mathtt{z}}}, \overline{{\underline{\mathtt{z}}}}, \overline{{\underline{\mathtt{z}}}}_0 )$ is a \textit{$\TTTTT $-pseudoalgebra} if $\overline{{\underline{\mathtt{z}}}}, \overline{{\underline{\mathtt{z}}}}_0 $ are invertible.

\item Morphisms: \textit{lax $\TTTTT$-morphisms} $\mathtt{f}:\mathtt{y}\to\mathtt{z} $ between lax $\TTTTT$-algebras $\mathtt{y}=(Y, \alg _ {{}_{\mathtt{y}}}, \overline{{\underline{\mathtt{y}}}}, \overline{{\underline{\mathtt{y}}}}_0 )$, $\mathtt{z}= (Z, \alg _ {{}_{\mathtt{z}}}, \overline{{\underline{\mathtt{z}}}}, \overline{{\underline{\mathtt{z}}}}_0 )$ are pairs $\mathtt{f} = (f, \left\langle \overline{\mathtt{f}}\right\rangle ) $ in which 
$f: Y\to Z  $
is a morphism in $\bbb $ and
$\left\langle   \overline{\mathtt{f}}\right\rangle :  \alg _ {{}_{\mathtt{z}}}\TTTTT (f) \Rightarrow f\alg _ {{}_{\mathtt{y}}}   $  
is a $2$-cell of $\bbb $ such that, defining $\widehat{\TTTTT (\left\langle   \overline{\mathtt{f}}\right\rangle )} : = \tttt ^{-1} _ {{}_{(f)(\alg _ {{}_{\mathtt{y}}}) }}  \TTTTT (\left\langle   \overline{\mathtt{f}}\right\rangle )\tttt _ {{}_{(\alg _ {{}_{\mathtt{z}}})(\TTTTT(f))  }}$, the equations

$$\xymatrix@C=2.4em{
&
\TTTTT ^2 Y\ar@{}[d]|-{\xLeftarrow{m_{{}_{f}}^{-1} } }\ar[ld]_ -{m _ {{}_{Y}} }\ar[r]^-{\TTTTT ^2(f) }
&
\TTTTT ^2 Z\ar[rd]^-{\TTTTT (\alg _ {{}_{\mathtt{z}}}) }\ar[ld]|-{m _ {{}_{Z}} }\ar@{}[dd]|-{\xLeftarrow{\,\,\overline{{\underline{\mathtt{z}}} }\, \,} }
&
\ar@{}[rdd]|-{=}
&
&
\TTTTT Z\ar[r]^{\alg _ {{}_{\mathtt{z}}}}\ar@{}[dd]|-{\xRightarrow{\widehat{\TTTTT (\left\langle   \overline{\mathtt{f}}\right\rangle)}}}
&
Z\ar@{}[d]|-{\xRightarrow{\left\langle\overline{\mathtt{f}}\right\rangle } }
&
\\
\TTTTT Y\ar[r]_{\TTTTT (f) }\ar[rd]_{\alg _ {{}_{\mathtt{y}}}}
&
\TTTTT Z\ar[rd]|-{\alg _ {{}_{\mathtt{z}}} }\ar@{}[d]|-{\xLeftarrow{\left\langle   \overline{\mathtt{f}}\right\rangle} }
&
&
\TTTTT Z\ar[ld]^-{\alg _ {{}_{\mathtt{z}}} }
&
\TTTTT ^2 Z\ar[ru]^-{\TTTTT (\alg _ {{}_{\mathtt{z}}} )}
&
&
\TTTTT Y\ar[lu]|-{\TTTTT (f) }\ar[r]^{\alg _ {{}_{\mathtt{y}}} }\ar@{}[d]|-{\xRightarrow{\,\overline{{\underline{\mathtt{y}}} }\, } }
&
Y\ar[lu]_ {f}
\\
&
Y\ar[r]_{f}
&
Z
&
&
&
\TTTTT ^2Y\ar[ru]|-{\TTTTT (\alg _ {{}_{\mathtt{y}}})}\ar[r]_{m _ {{}_{Y}} }\ar[lu]^{\TTTTT ^2(f)}
&
\TTTTT Y\ar[ru]_-{\alg _ {{}_{\mathtt{y}}} }
&
}$$
$$\xymatrix{ 
Y\ar[rr]^{f}\ar[d]_{\eta _ {{}_{Y}} }\ar@{}[rd]|-{\xLeftarrow{\eta _ {{}_{f}}^{-1}} }
&
&
Z\ar[ld]_-{\eta _ {{}_{Z}} }\ar@{=}[dd]\ar@{}[rdd]|-{=}
&
&
Y\ar[ld]_{\eta _ {{}_{Y}} }\ar@{=}[dd]
&
\\
\TTTTT Y\ar[r]_{\TTTTT (f)}\ar[d]_-{\alg _ {{}_{\mathtt{y}}} }\ar@{}[rd]|-{\xLeftarrow{\left\langle   \overline{\mathtt{f}}\right\rangle}}
&
\TTTTT Z\ar@{}[r]|-{\xLeftarrow{\overline{{\underline{\mathtt{z}}}}_0} }\ar[rd]_-{\alg _ {{}_{\mathtt{z}}} }
&
&
\TTTTT Y\ar@{}[r]|-{\xLeftarrow{\overline{{\underline{\mathtt{y}}}}_0 } }\ar[rd]_{\alg _ {{}_{\mathtt{y}}}}
&
&
\\
Y\ar[rr]_-{f}
&
&
Z
&
&
Y\ar[r]_{f}
&
Z
}$$
hold. Recall that a lax $\TTTTT $-morphism $\mathtt{f} = (f,  \left\langle \overline{\mathtt{f}}\right\rangle ) $ is called a $\TTTTT $-pseudomorphism if $\left\langle   \overline{\mathtt{f}}\right\rangle$ is an invertible $2$-cell. 

\item $2$-cells: a \textit{$\TTTTT $-transformation}  $\mathfrak{m} : \mathtt{f}\Rightarrow\mathtt{h} $ between lax $\TTTTT $-morphisms $\mathtt{f} = (f,\left\langle   \overline{\mathtt{f}}\right\rangle )$, $\mathtt{h} = (h, \left\langle\overline{\mathsf{h}}\right\rangle )$ is a $2$-cell $\mathfrak{m} : f\Rightarrow h $ in $\bbb $ such that the equation below holds.
$$\xymatrix{  \TTTTT Y\ar@/_4ex/[dd]_{\TTTTT (f) }
                    \ar@{}[dd]|{\xRightarrow{\TTTTT(\mathfrak{m}) } }
                    \ar@/^4ex/[dd]^{\TTTTT (h) }
										\ar[rr]^{\alg _ {{}_{\mathtt{y}}} } && 
										 Y\ar[dd]^{h }    
&&
\TTTTT Y \ar[rr]^{ \alg _ {{}_{\mathtt{y}}} }\ar[dd]_-{\TTTTT (f) }  &&
Y\ar@/_5ex/[dd]_{f}
                    \ar@{}[dd]|{\xRightarrow{\mathfrak{m}  } }
                    \ar@/^5ex/[dd]^{h }
\\
&\ar@{}[r]|{\xRightarrow{\hskip .2em \left\langle\overline{\mathsf{h}}\right\rangle  \hskip .2em } }   &
 &=& 
&\ar@{}[l]|{\xRightarrow{\hskip .2em \left\langle \overline{\mathtt{f}}\right\rangle  \hskip .2em } }  & 
\\
 \TTTTT Z\ar[rr]_ {\alg _ {{}_{\mathtt{z}}} } &&  Z
 &&
\TTTTT Z\ar[rr]_ {\alg _ {{}_{\mathtt{z}}} } && Z	 }$$

\end{itemize}

The compositions are defined in the obvious way and these definitions make $\mathsf{Lax}\textrm{-}\TTTTT\textrm{-}\Alg _ {\ell } $ a $2$-category. The full sub-$2$-category of the $\TTTTT$-pseudoalgebras of $\mathsf{Lax}\textrm{-}\TTTTT\textrm{-}\Alg _ {\ell } $ is denoted by $\mathsf{Ps}\textrm{-}\TTTTT\textrm{-}\Alg _ {\ell } $. Also, we consider the locally full sub-$2$-category  $\mathsf{Ps}\textrm{-}\TTTTT\textrm{-}\Alg  $ consisting of $\TTTTT $-pseudoalgebras and $\TTTTT$-pseudomorphisms. Finally, the inclusion is denoted by $\widehat{\ell }  : \mathsf{Ps}\textrm{-}\TTTTT\textrm{-}\Alg \to \mathsf{Lax}\textrm{-}\TTTTT\textrm{-}\Alg _ {\ell } $. 
\end{defi}

On one hand, if $\TTTTT = (\TTTTT,  m  , \eta , \mu,   \iota, \tau ) $  is a pseudomonad on $\bbb $, 
then $\TTTTT $ induces the \textit{Eilenberg-Moore biadjunction} $(L ^\TTTTT \dashv U^\TTTTT , \eta , \varepsilon ^\TTTTT , s ^\TTTTT , t ^\TTTTT ) $,
in which $L^\TTTTT , U ^\TTTTT $ are defined by
\small
\begin{eqnarray*}
\begin{aligned}
U^\TTTTT : \mathsf{Ps}\textrm{-}\TTTTT\textrm{-}\Alg & \to \bbb\\ 
(Z, \alg _ {{}_{\mathtt{z}}}, \overline{{\underline{\mathtt{z}}}}, \overline{{\underline{\mathtt{z}}}}_0 )&\mapsto Z\\
(f, \left\langle \overline{\mathtt{f}}\right\rangle ) &\mapsto f\\
\mathfrak{m} &\mapsto \mathfrak{m}
\end{aligned}
\qquad\qquad
\begin{aligned}
L^\TTTTT : \bbb &\to \mathsf{Ps}\textrm{-}\TTTTT\textrm{-}\Alg\\
             Z &\mapsto \left( \TTTTT (Z), m _ {{}_{Z}}, \mu _ {{}_{Z}}, \iota ^{-1} _ {{}_{Z}}\right)\\
						f &\mapsto \left(\TTTTT(f), m _ {{}_{f}}^{-1}\right)\\
						\mathfrak{m} &\mapsto \TTTTT(\mathfrak{m})						
\end{aligned}
\end{eqnarray*}
\normalsize
On the other hand, this biadjunction induces a pseudocomonad on $\mathsf{Ps}\textrm{-}\TTTTT\textrm{-}\Alg $, called the
\textit{Eilenberg-Moore pseudocomonad} and denoted herein by 
$\underline{\TTTTT } $ (see Remark 5.3 and Lemma 5.4 of \cite{MR3491845}). The underlying pseudofunctor of this pseudocomonad
can be extended to $\mathsf{Lax}\textrm{-}\TTTTT\textrm{-}\Alg _ {\ell } $ by composing the forgetful $2$-functor $\mathsf{Lax}\textrm{-}\TTTTT\textrm{-}\Alg _ {\ell }  \to \bbb $ with $\widehat{\ell } L ^\TTTTT $. By abuse of language, we denote this extension (and its restriction to $\mathsf{Ps}\textrm{-}\TTTTT\textrm{-}\Alg _ {\ell } $) by $\underline{\TTTTT } $ as well. It should be observed that
$\underline{\TTTTT }\, \mathtt{z}  $ is the free $\TTTTT $-pseudoalgebra on the underlying object of the lax $\TTTTT $-algebra $\mathtt{z} $,
while $\underline{\TTTTT } (f, \left\langle \overline{\mathtt{f}}\right\rangle ) =\left(\TTTTT(f), m _ {{}_{f}}^{-1}\right) $.

\begin{rem}[Counit]
Given a lax $\TTTTT$-algebra $(Z, \alg _ {{}_{\mathtt{z}}}, \overline{{\underline{\mathtt{z}}}}, \overline{{\underline{\mathtt{z}}}}_0 )$, we denote by $\varepsilon ^\TTTTT _ {{}_{\mathtt{z}}} $  the lax $\TTTTT$-morphism defined by:
$$\varepsilon ^\TTTTT _ {{}_{\mathtt{z}}} \coloneqq (\alg _ {{}_{\mathtt{z}}}, \overline{{\underline{\mathtt{z}}}}): 
\underline{\TTTTT }\mathtt{z}\to \mathtt{z} ,  $$ 
while, given a $\TTTTT $-pseudomorphism $\mathtt{f} = (f, \left\langle \overline{\mathtt{f}}\right\rangle ) $, 
we denote by $\varepsilon ^\TTTTT _ {{}_{\mathtt{f} }} $  the $\TTTTT $-transformation defined by $\left\langle \overline{\mathtt{f}}\right\rangle ^{-1} $. It should be noted that $\varepsilon ^\TTTTT $ restricted
to the pseudoalgebras is actually the counit of the pseudocomonad $\underline{\TTTTT }$.
\end{rem}

\section{Non-canonical isomorphisms}\label{Prenormal}

In this section, we prove our main result. We start with: 

\begin{defi}[$\mathtt{f}$-isomorphism]\label{noncanonicalisomorphism}
Let $\TTTTT $ be a pseudomonad on a $2$-category $\bbb $. Assume that 
$\mathtt{f} = (f, \left\langle \overline{\mathtt{f}}\right\rangle ):\mathtt{y}\to\mathtt{z}  $
is a lax $\TTTTT $-morphism. If it exists,  an invertible $\TTTTT $-transformation 
$$\xymatrix@C=5em{
\underline{\TTTTT }\mathtt{y}
\ar[r]^{\underline{\TTTTT } (\mathtt{f} ) }
\ar[d]_{\varepsilon ^\TTTTT _ {{}_{\mathtt{y} }}  }
\ar@{}[rd]|-{\xLeftarrow{\hskip 1em\psi  \hskip 0.9em } }
&
\underline{\TTTTT }\mathtt{z}
\ar[d]^{\varepsilon ^\TTTTT _ {{}_{\mathtt{z} }}  }
\\
\mathtt{y}
\ar[r]_{\mathtt{f}}
&
\mathtt{z}
} $$
is called an  \textit{$\mathtt{f}$-isomorphism}.
\end{defi} 

Roughly, these $\mathtt{f}$-isomorphisms play the role of the \textit{non-canonical isomorphisms} in the examples given in the introduction. It should be noted that, if the canonical transformation is invertible, then it is an $\mathtt{f}$-isomorphism as well. 
That is to say,  the first basic result about $\mathtt{f} $-isomorphisms is the following: for each $\TTTTT $-pseudomorphism $\mathtt{f} = (f, \left\langle \overline{\mathtt{f}}\right\rangle ) $ between lax $\TTTTT $-algebras $\mathtt{y} $ and $\mathtt{z}$,  
$$\left\langle \overline{\mathtt{f}}\right\rangle = (\varepsilon ^\TTTTT _ {{}_{\mathtt{f} }})^{-1} : \varepsilon ^\TTTTT _ {{}_{\mathtt{z} }} \cdot \underline{\TTTTT } (\mathtt{f} )\Longrightarrow \mathtt{f}\cdot \varepsilon ^\TTTTT _ {{}_{\mathtt{y} }} $$ is an $\mathtt{f}$-isomorphism. Theorem \ref{MAIN} gives the reciprocal to this fact for lax $\TTTTT $-morphisms between $\TTTTT$-pseudoalgebras.

\begin{theo}[Main Theorem]\label{MAIN}
Let $\TTTTT $ be a pseudomonad on a $2$-category $\bbb $. A lax $\TTTTT $-morphism $\mathtt{f}: \mathtt{y}\to\mathtt{z} $ between $\TTTTT $-pseudoalgebras is a $\TTTTT $-pseudomorphism if and only if there is an $\mathtt{f} $-isomorphism.
\end{theo}

\begin{proof}
It remains only to prove that a lax $\TTTTT $-morphism $\mathtt{f}$ is a $\TTTTT $-pseudomorphism provided that there is an 
$\mathtt{f} $-isomorphism. We assume that the structures of the pseudomonad $\TTTTT $, the lax $\TTTTT$-morphism $\mathtt{f}$ and the 
$\TTTTT $-pseudoalgebras  $\mathtt{y} $ and $\mathtt{z} $ are given as in Definition
\ref{laxalgebras}. 

Assume that $\psi : \varepsilon ^\TTTTT _ {{}_{\mathtt{z} }} \cdot \underline{\TTTTT } (\mathtt{f} )\Longrightarrow \mathtt{f}\cdot \varepsilon ^\TTTTT _ {{}_{\mathtt{y} }} $
is an invertible $\TTTTT $-transformation. By the definition of $\TTTTT $-transformation, we conclude that
$$\xymatrix{
\TTTTT  Z
\ar[rr]|-{\alg _ {{}_{\mathtt{z}}} }
\ar@{}[rd]|-{\xRightarrow{\hskip 4pt\overline{{\underline{\mathtt{z}}}} \hskip 8pt} }
&
&
Z
\ar@{}[dd]|-{\xRightarrow{\hskip 7pt \psi \hskip 7pt } }
&
&
&
\TTTTT Z 
\ar@{}[dd]|-{\xRightarrow{\hskip 6pt\widehat{\TTTTT (\psi ) }\hskip 6pt} }
\ar[rr]|-{\alg _ {{}_{\mathtt{z}}} }
&
&
Z
\\
\TTTTT ^2 Z 
\ar[u]^-{\TTTTT (\alg _ {{}_{\mathtt{z}}} ) }
\ar[r]|-{m _ {{}_{Z}} }
\ar@{}[rd]|-{\xRightarrow{ m _ {{}_{f}}^{-1}  } }
&
\TTTTT Z
\ar[ru]|-{\alg _ {{}_{\mathtt{z}}} }
&
&
Y
\ar[lu]|-{f}
\ar@{}[r]|-{=}
&
\TTTTT ^2 Z
\ar[ru]|-{\TTTTT (\alg _ {{}_{\mathtt{z}}} ) }
&
&
\TTTTT Y
\ar@{}[ru]|-{\xRightarrow{\hskip 7pt \left\langle \overline{\mathtt{f}}\right\rangle\hskip 6pt } }
\ar[r]|-{\alg _ {{}_{\mathtt{y}}} }
\ar[lu]|-{\TTTTT (f) }
\ar@{}[rd]|-{\xRightarrow{\hskip 7pt\overline{{\underline{\mathtt{y}}}} \hskip 7pt} }
&
Y
\ar[u]_-{f}
\\
\TTTTT ^2 Y
\ar[rr]|-{m _ {{}_{Y}} }
\ar[u]^-{\TTTTT ^2 (f) }
&
&
\TTTTT Y
\ar[lu]|-{\TTTTT (f) }
\ar[ru]|-{\alg _ {{}_{\mathtt{y}}}}
&
&
&
\TTTTT ^2 Y 
\ar[rr]|-{m _ {{}_{Y}} }
\ar[ru]|-{\TTTTT (\alg _ {{}_{\mathtt{y}}} ) }
\ar[lu]|-{\TTTTT ^2 ( f ) }
&
&
\TTTTT Y
\ar[u]_-{\alg _ {{}_{\mathtt{y}}}}
}$$
holds in $\bbb $,
in which $\widehat{\TTTTT (\psi )} \coloneqq \tttt ^{-1} _ {{}_{(f)(\alg _ {{}_{\mathtt{y}}}) }}  \TTTTT (\psi )\tttt _ {{}_{(\alg _ {{}_{\mathtt{z}}})(\TTTTT(f))  }}$. Since we know that all the $2$-cells above but $\left\langle \overline{\mathtt{f}}\right\rangle $ are invertible, after composing with the appropriate inverses in both sides of the equation, 
we conclude that the horizontal composition 
$ \left\langle \overline{\mathtt{f}}\right\rangle \ast \id _ {{}_{\TTTTT (\alg _ {{}_{\mathtt{y}}}) }} $ is invertible as well.

Therefore, defining  $\widehat{\TTTTT \left( \overline{{\underline{\mathtt{y}}}}_0 \right)} := 
\tttt ^{-1} _ {{}_{\left(\alg _ {{}_{\mathtt{y}}}\right)\left(\eta _{{}_{Y}} \right) }}  
\TTTTT \left(\overline{{\underline{\mathtt{y}}}}_0  \right)\tttt _ {{}_{Z}}$, we conclude that the left hand of the equality 
$$
\xymatrix{
&
\TTTTT Y
\ar[d]|-{\TTTTT (\eta _{{}_{Y}} )}
\ar@{=}@/_5pc/[dd]
\ar@{=}@/^5pc/[dd]
&
&
&
&
\\
&
\TTTTT ^2 Y
\ar@{}[l]|-{\xLeftarrow{\widehat{\TTTTT \left( \overline{{\underline{\mathtt{y}}}}_0 \right)} ^{-1} } }
\ar@{}[r]|-{\xLeftarrow{\widehat{\TTTTT \left( \overline{{\underline{\mathtt{y}}}}_0 \right)}  } }
\ar[d]|-{\TTTTT (\alg _ {{}_{\mathtt{y}}}) }
&
&
&
&
\\
&
\TTTTT Y 
\ar[rd]^{\TTTTT (f) }
\ar[ld]_{\alg _ {{}_{\mathtt{y}}}}
&
&
&
\TTTTT Y
\ar[rd]^{\TTTTT (f) }
\ar[ld]_{\alg _ {{}_{\mathtt{y}}}}
&
\\
Y
\ar[rd]_{f}
\ar@{}[rr]|-{\xLeftarrow{\left\langle \overline{\mathtt{f}}\right\rangle } }
&
&
\TTTTT Z
\ar@{}[r]|-{=}
\ar[ld]^{\alg _ {{}_{\mathtt{z}}} }
&
Y
\ar[rd]_{f}
\ar@{}[rr]|-{\xLeftarrow{\left\langle \overline{\mathtt{f}}\right\rangle } }
&
&
\TTTTT Z
\ar[ld]^{\alg _ {{}_{\mathtt{z}}} }
\\
&
Z
&
&
&
Z
&
} $$
is a (vertical) composition of invertible $2$-cells and, hence, itself invertible.
\end{proof}

\begin{rem}
By doctrinal adjunction~\cite{MR0360749}, we conclude that a lax $\TTTTT $-morphism $\mathtt{f} = (f, \left\langle \overline{\mathtt{f}}\right\rangle ) $
between $\TTTTT $-pseudoalgebras has a right adjoint in $\mathsf{Lax}\textrm{-}\TTTTT\textrm{-}\Alg _ {\ell } $  if and only if $f$ has a right adjoint in the base $2$-category $\bbb $ and  
there is an $\mathtt{f} $-isomorphism.
\end{rem}

\section{Examples}\label{Applications}
In this section, we state some examples. Instead of giving a definitive answer to every particular case, Theorem \ref{MAIN} gives a general 
setting to the problems of non-canonical isomorphisms, giving a general procedure for studying them. 
Namely, in each \textit{non-canonical isomorphism problem}, we can firstly show
 how this problem can be framed in our context and, then, show that such non-canonical isomorphism actually defines a
suitable $\mathtt{f}$-isomorphism.

For example, the result on braided monoidal categories of \cite{MR2864762} is 
a particular instance of Theorem \ref{MAIN}. Firstly, we establish the direct corollary of Theorem \ref{MAIN} on monoidal categories.  
We denote by $\F $ the \textit{free monoid $2$-monad} on $\Cat $.
The underlying $2$-functor of this $2$-monad is given by 
$\F (\AAA ) = \displaystyle\coprod _ {n = 0 }^\infty\AAA ^n $, while each component 
of the multiplication
 is induced by the identities $\AAA ^t \to \AAA ^t $ for each $t$. The $2$-category $\mathsf{Ps}\textrm{-}\F\textrm{-}\Alg _ {\ell } $
is known to be the $2$-category of monoidal categories, monoidal functors and monoidal transformations~\cite{MR1007911, 2016arXiv160703087L}.
Recall that $\F $-pseudomorphisms are called \textit{strong monoidal functors}.

\begin{rem}
We adopt the biased definition of monoidal category.
 Given a monoidal category 
$\MMM = (\MMM _ 0 , \otimes _\MMM , I _ \MMM ) $,  the monoidal product
$\otimes _{\underline{\F } \MMM } :  \left(\underline{\F } \MMM\right)_0\times \left(\underline{\F } \MMM\right)_0
\to   \left(\underline{\F } \MMM\right)_0 $ of the strict monoidal category $\underline{\F } \MMM $
is defined firstly by taking the isomorphism $\displaystyle\left(\coprod  _ {k = 0 }^\infty \MMM _ 0 ^k\right)\times \left(\coprod  _ {j = 0 }^\infty \MMM _ 0 ^j\right)\cong \coprod  _ {k,j\in\mathbb{N}  } \MMM _ 0 ^{k+j} $ and composing with the morphism induced by the
canonical inclusions of the coproduct $\displaystyle \coprod  _ {k = 0 }^\infty \MMM _ 0 ^{k+t}\to\coprod  _ {n = 0 }^\infty \MMM _ 0 ^{n} $ for each $t$.
In other words, an object of $\left(\underline{\F } \MMM\right)_0\times \left(\underline{\F } \MMM\right)_0 $ is an ordered pair of words of objects in $\MMM _ 0 $, while the tensor product is just the word obtained by juxtaposition. The empty word is the identity of the monoidal structure of  $\underline{\F } \MMM $.

It should be noted that the component $\varepsilon ^\F _ {{}_{\MMM }}: \underline{\F } \MMM \to \MMM $ is a strong monoidal functor. We consider
that its underlying functor gives the monoidal product of the objects in the word respecting the order, that is to say, it is defined inductively by 
\begin{eqnarray*}
\alg _{{}_{\MMM }} () &:= & I _ {\MMM };\\
\alg _{{}_{\MMM }} (x_1) &:= & x_1;\\
\alg _ {{}_{\MMM }}  (x_1 , \ldots , x_{n+1} ) &:= & \alg _ {{}_{\MMM }}  (x_1 , \ldots , x_{n} )\otimes _\MMM   \alg _ {{}_{\MMM }}  (x_{n+1} )
\end{eqnarray*}  
in which $() $ denotes the object of $\MMM _ 0 ^0 $, which is the identity object of $\underline{\F } \MMM $, and $(x_1, \ldots , x_n ) $ denotes an object of $\MMM _0 ^n $. Finally, for each monoidal functor $\mathtt{f} = (f, \left\langle \overline{\mathtt{f}}\right\rangle ) $, 
$\underline{\F}(\mathtt{f}) $ is the strict monoidal functor defined pointwise by $f$.  
\end{rem}

\begin{coro}[Strong monoidal functors]\label{Corolario para Monoidais}
Let $\mathtt{f} = (f, \left\langle \overline{\mathtt{f}}\right\rangle ) : \MMM\to\NNN $ be a monoidal functor. There is an invertible monoidal transformation $\psi :\varepsilon ^\F _ {{}_{\NNN }} \cdot \underline{\F } (\mathtt{f} )\Longrightarrow \mathtt{f}\cdot \varepsilon ^\F _ {{}_{\MMM }}   $ if and only if $\mathtt{f} $ is a strong monoidal functor or, in other words, $\left\langle \overline{\mathtt{f}}\right\rangle $ 
is invertible.
\end{coro}

Given a
braided monoidal category $\MMM = (\MMM _0 , \otimes _ \MMM , I _\MMM , \lambda _ \MMM ) $, we have an induced strong 
monoidal functor
 $\MMM\times \MMM\to \MMM $
 whose underlying functor
is $\otimes _\MMM $ and the structure maps are given by the isomorphisms 
$w\otimes x\otimes y\otimes z\to w\otimes y\otimes x\otimes z $ and $ I _ \MMM \otimes I _ \MMM \to I _ \MMM $
induced respectively by the braiding $\lambda _ \MMM $ and the action of the identity. In Theorem \ref{Teorema do Lack},
we denote this strong monoidal functor just by $\otimes _\MMM $.

Recall that a monoidal functor $\mathtt{f} = (f, \left\langle \overline{\mathtt{f}}\right\rangle )$ is called normal if the component of 
$\left\langle \overline{\mathtt{f}}\right\rangle $ on the empty word, denoted below by $\left\langle \overline{\mathtt{f}}\right\rangle _{()} $,  is invertible.

\begin{coro}[\cite{MR2864762}]\label{Teorema do Lack}
Let $\MMM , \NNN $ be braided monoidal categories, and $\mathtt{f}= (f, \left\langle \overline{\mathtt{f}}\right\rangle ):\MMM \to \NNN  $ a normal monoidal functor.
If we have a invertible monoidal transformation
$$\xymatrix{
\MMM\times \MMM 
\ar[r]^-{\mathtt{f}\times \mathtt{f} }
\ar[d]_-{\otimes _ {{}_{\MMM }}  }
\ar@{}[rd]|-{\xLeftarrow{\hskip 1em\varphi\hskip 1em } }
&
\NNN\times\NNN
\ar[d]^-{\otimes _ {{}_{\NNN }}  }
\\
\MMM
\ar[r]_-{\mathtt{f} }
&
\NNN
}$$
then $\mathtt{f} $ is a strong monoidal functor.
\end{coro}
\begin{proof}
In fact, from $\varphi $, we define an invertible monoidal transformation $\psi :\varepsilon ^\F _ {{}_{\NNN }} \cdot \underline{\F } (\mathtt{f} )\Longrightarrow \mathtt{f}\cdot \varepsilon ^\F _ {{}_{\MMM }}   $
as in Corollary \ref{Corolario para Monoidais} inductively as follows:
\begin{eqnarray*}
\psi _ {{}_{()}} &:= & \left\langle \overline{\mathtt{f}}\right\rangle _ {()} \\
\psi _ {{}_{(x_1)}} &:= & \id _ {{}_{f(x_1) }} \\
\psi _ {{}_{(x_1, \ldots , x_n , x_{n+1})}} &:= & \varphi _ {{}_{(f(x_1)\otimes _\NNN\cdots \otimes _ \NNN f(x_{n}),  x_{n+1} )}} \cdot \left( \psi _ {{}_{(x_1, \ldots , x_n )}}\otimes _ \MMM \id _ {{}_{f(x_ {n+1})}} \right).
\end{eqnarray*}
\end{proof}

As noted therein, Corollary \ref{Teorema do Lack} encompasses all the common parts of the examples presented in \cite{MR2864762}, mentioned in the introduction. This includes the non-canonical isomorphism problem of preservation of coproducts studied by Caccamo-Winskel~\cite{MR2207108}.

\subsection{Kock-Z\"{o}berlein pseudomonads}

Kock-Z\"{o}berlein pseudomonads provide examples in which the setting of Theorem \ref{MAIN} is simpler. 
This is mostly due to the fact that  Kock-Z\"{o}berlein pseudomonads~\cite{MR1359690, MR1432190, MR1476422} satisfy the hypothesis of:

\begin{coro}\label{lff}
Assume that $\TTTTT $ is a pseudomonad on $\bbb$ such that $\mathsf{Ps}\textrm{-}\TTTTT\textrm{-}\Alg _ {\ell}\to \bbb $
is locally fully faithful. Assume that $\mathtt{f}= (f, \left\langle \overline{\mathtt{f}}\right\rangle ):\mathtt{y}\to\mathtt{z} $ 
is a lax $\TTTTT $-morphism between the $\TTTTT$-pseudoalgebras $\mathtt{y}=(Y, \alg _ {{}_{\mathtt{y}}}, \overline{{\underline{\mathtt{y}}}}, \overline{{\underline{\mathtt{y}}}}_0 )$ and $\mathtt{z}= (Z, \alg _ {{}_{\mathtt{z}}}, \overline{{\underline{\mathtt{z}}}}, \overline{{\underline{\mathtt{z}}}}_0 )$. There is an invertible $2$-cell
$$\xymatrix@C=5em{
\TTTTT Y
\ar[r]^{\TTTTT  ( f ) }
\ar[d]_{\alg _ {{}_{\mathtt{y} }}  }
\ar@{}[rd]|-{\xLeftarrow{\hskip 1em\psi  \hskip 0.9em } }
&
\TTTTT Z
\ar[d]^{\alg _ {{}_{\mathtt{z} }}  }
\\
Y
\ar[r]_{f}
&
Z
} $$
in $\bbb $ if and only if $\mathtt{f} $ is a $\TTTTT $-pseudomorphism.
\end{coro}

\begin{proof}
By hypothesis, every such an invertible $2$-cell gives an $\mathtt{f}$-isomorphism. Therefore the result follows from Theorem \ref{MAIN}.
\end{proof}

Cocompletion pseudomonads~\cite{MR1776428} are examples of Kock-Z\"{o}berlein pseudomonads. Hence, in particular, Corollary
\ref{lff} shows how Theorem \ref{MAIN} encompasses the non-canonical isomorphism problem of preservation 
of conical colimits studied in \cite{MR2207108} and, more generally, the non-canonical isomorphisms for 
preservation of weighted colimits~\cite{MR1765587, MR1776428}.

\begin{rem}[Binary coproducts]
We exemplify how the problem of preservation of conical colimits is framed in our setting by showing the case of
binary coproducts. In order to do so, we consider the $2$-monad of the 
free cocompletion by binary coproducts, called herein $\mathrm{Fam}_{\mathrm{f}} $. 

Recall that the objects of 
$\mathrm{Fam}_{\mathrm{f}} (\AAA ) $ are non-empty finite families of objects in $\AAA $, which can
be seen as lists of objects $(x_1 , \ldots , x_{n} )$. A morphism 
$(x_1 , \ldots , x_{n} )\to (y_1 , \ldots , y_{m} )$ is a list $(t _0, t_1 , \ldots , t_ n)$ in which 
$$t _ 0: \left\{1, \ldots , n\right\}\to \left\{1, \ldots , m\right\} $$
is a function and, for $j>0$, $t_j : x_j\to y_ {{}_{t_0 (j)} } $  is a morphism of $\AAA $. 

Assuming that $\AAA $ has binary coproducts, we denote the structure by $\alg _\AAA :  \mathrm{Fam}_{\mathrm{f}} (\AAA ) \to \AAA $. It should be noted that $\alg _\AAA  (t _0, t_1 , \ldots , t_ n)$ is the composition of the arrows below, in which the second arrow is induced by the 
morphisms $ y_{{}_{t_0(j)}}\to \coprod _{i=1}^m y_i $ of the universal cocone.
$$\xymatrix{\displaystyle\coprod _{j=1}^n x_j \ar[rr]^-{\coprod _{j=1}^n (t_j) }&&  \displaystyle\coprod _{j=1}^n y_ {{}_{t_0 (j)}}
\ar[r] &  \displaystyle\coprod _{i=1}^m y_i  } $$

Let $F: \AAA \to \BBB $ be a functor between categories with binary coproducts. On one hand, assuming that
$\AAA $ has initial object $O$ and that $F$ preserves it, given an isomorphism $\alpha _ {(x_1, x_2)}: F(x_1)\sqcup F(x_2)\to F(x_1\sqcup x_2) $ natural in $x_1 $ and $x_2$, we can define $\alpha ' _ {(y_1)} $ to be the composition of the arrows
$$\xymatrix{F(y_1) \ar[r]& F(y_1)\sqcup F(O)\ar[rr]^{\alpha _ {(y_1, O)} } && F(y_1\sqcup O ) \ar[r] & F(y_1)    }$$
in which $F(y_1\sqcup O )\to F(y_1) $ is the image of the inverse of the canonical morphism
$y_1\to y_1\sqcup O $ and the other non-labeled arrow is the canonical one. Then, we put 
$\alpha ' _ {(x_1, x_2)}\coloneqq \alpha _ {(x_1, x_2)} $ and  
we can define 
$\alpha ' _ {(x_1, \ldots , x_n)} $ inductively. This gives a natural isomorphism 
$$\alpha ' : \alg _\BBB\circ \mathrm{Fam}_{\mathrm{f}}(F)\longrightarrow  F\circ \alg _\AAA $$
implying, by Corollary \ref{lff}, that $F$ preserves finite coproducts. This proves Theorem 3.3 of \cite{MR2207108}.

On the other hand, clearly there are functors that preserve binary coproducts but do not preserve initial objects. 
The obvious example is the inclusion of the codomain $d^0: \mathsf{1}\to\mathsf{2} $, in which $\mathsf{1} $ is the
terminal category with only the object $\mathsf{0} $ and $\mathsf{2} $ is the category $\mathsf{0}\to\mathsf{1}$.

Yet, the existence of a natural isomorphism $\alpha _ {(x_1, x_2)}: F(x_1)\sqcup F(x_2)\to F(x_1\sqcup x_2) $
does not suffice to construct an $F$-isomorphism as above, or, equivalently, to get the preservation of
binary coproducts. A counterexample for preservation of
binary products is given at the end of Section 2 of \cite{MR2207108}. 
The general (dual) idea of that counterexample is to consider a functor $G:\mathsf{1}\to \AAA $.
A natural isomorphism is just an isomorphism $ G(\mathsf{0} )\sqcup G(\mathsf{0} )\cong G(\mathsf{0} ) $, while
the canonical comparison is the codiagonal (morphism induced by the identities). In the category of sets, the
codiagonal
$\mathbb{N}\sqcup \mathbb{N}\to \mathbb{N} $ is not an isomorphism, although
there is an obvious isomorphism $\mathbb{N}\sqcup \mathbb{N}\cong \mathbb{N} $.

However, for instance, it is easy to see that the existence of 
 an $F$-isomorphism is equivalent to the existence of a natural isomorphism $\beta: F\longrightarrow F $
such that, for each pair $(x_1 , x_2) $, the morphism $F(x_1)\sqcup F(x_2)\to F(x_1\sqcup x_2 ) $
induced by the morphisms   
$$F(x_1)\stackrel{\beta _ {{}_{x_1}}}{\rightarrow} F(x_1)\rightarrow F(x_1\sqcup x_2) \quad\mbox{and}\quad F(x_2)\stackrel{\beta _ {{}_{x_2}}}{\rightarrow} F(x_2)\rightarrow F(x_1\sqcup x_2)$$
is invertible, in which the unlabeled arrows are the images of the canonical morphisms. 

\end{rem}

\bibliographystyle{plain}
\bibliography{references}

 \end{document}